\documentclass[12pt]{amsart}
\usepackage{amsmath,amssymb,amsfonts,enumerate,amsthm}

\headsep=1cm \numberwithin{equation}{section}
\newcommand{\fa}{\mathfrak{a}}
\newcommand{\fb}{\mathfrak{b}}
\newcommand{\fm}{\mathfrak{m}}
\newcommand{\fp}{\mathfrak{p}}
\newcommand{\fu}{\mathfrak{U}}
\newcommand{\fP}{\mathfrak{P}}
\newcommand{\fq}{\mathfrak{q}}

\newcommand{\ft}{\mathfrak{T}}
\newtheorem{thm}{Theorem}[section]
\newtheorem{cor}[thm]{Corollary}
\newtheorem{lem}[thm]{Lemma}
\newtheorem{prop}[thm]{Proposition}
\newtheorem{defn}[thm]{Definition}
\newtheorem{defn rem}[thm]{Definition and Remark}

\newtheorem{rem rem}[thm]{Reminder and Remark}
\newtheorem{exam}[thm]{Example}
\newtheorem{rem}[thm]{Remark}
\newtheorem{rem defn}[thm]{Remark and Definition}
\newtheorem{remi}[thm]{Reminder}
\newtheorem{nota}[thm]{Notation}
\newtheorem{con not}[thm]{Convention and Notation}
\numberwithin{equation}{section}

\begin{document}

\bibliographystyle{amsplain}

\author{Markus Brodmann}
\address{University of Z\"urich, Mathematics Institute, Winterthurerstrasse 190, 8057 Z\"urich. }
\email{brodmann@math.uzh.ch}

\author{Maryam Jahangiri}
\address{School of Mathematics and Computer Sciences, Damghan University, Damghan, Iran-and-School of Mathematics, Institute for research in Fundamental
Sciences (IPM), P.O.Box 19395-5746, Tehran, Iran.}
\email{jahangiri.maryam@gmail.com}

\thanks{\today }
\thanks{ The second author was in part supported by a grant from IPM (No. 89130115).}


\subjclass[2000]{}

\title[TAME LOCI OF CERTAIN LOCAL COHOMOLOGY MODULES]
 {TAME LOCI OF CERTAIN LOCAL COHOMOLOGY MODULES}
\begin{abstract} Let $M$ be a finitely generated graded module over a Noetherian homogeneous ring
$R = \bigoplus_{n \in \mathbb{N}_0}R_n$. For each $i \in \mathbb{N}_0$ let $H^i_{R_{+}}(M)$ denote the $i$-th local cohomology
module of $M$ with respect to the irrelevant ideal $R_+ = \bigoplus_{n > 0} R_n$ of $R$, furnished with its natural grading.
We study the tame loci $\ft^i(M)^{\leq 3}$ at level $i \in \mathbb{N}_0$ in codimension $\leq 3$ of $M$, that is the sets
of all primes $\fp_0 \subset R_0$ of height $\leq 3$ such that the graded $R_{\fp_0}$-modules $H^i_{R_{+}}(M)_{\fp_0}$ are tame.

\noindent
\end{abstract}
\maketitle


\section{ Introduction}

Throughout this note let  $ R=\bigoplus _{n \geq0} R_{n}$ be a
homogeneous Noetherian ring. So, $R$ is an $\mathbb{N}_{0} $-graded
$R_{0}$-algebra and $R= R_{0}[l_{1}, ..., l_{r}]$ with finitely many
elements $l_{1}, ..., l_{r} \in R_{1}$.
Moreover, let $R_{+}:=\bigoplus _{n>0}R_{n}$ denote the irrelevant
ideal of $R$ and let $M$ be a finitely generated graded $R$-module.
For each $i\in\mathbb{N}_{0}$ let $H^{i}_{R_{+}}(M)$ denote the
$i$-th local cohomology module of $M$
with respect to $R_{+}$. It is well known, that the $R$-module $H^{i}_{R_{+}}(M)$ carries a
natural grading and that the graded components $H^{i}_{R_{+}}(M)_{n}$ are finitely generated
$R_{0}$-modules which vanish for all $n\gg0$ (s. \cite{BS}, \S15 for example). So,
the $R_{0}$-modules $H^{i}_{R_{+}}(M)_{n}$ are asymptotically trivial if $n\rightarrow +\infty$.

On the other hand a rich variety of phenomena occurs for the modules $H^i_{R_{+}}(M)_n$ if $i \in \mathbb{N}_0$
is fixed and $n \rightarrow -\infty$. So, it is quite natural to investigate the \textit{ asymptotic behaviour
of cohomology}, e.g.the mentioned phenomena (s. \cite{B2}).

One basic question in this respect is to ask for the \textit{asymptotic stability of associated primes},
more precisely the question, whether for given $i \in \mathbb{N}_0$ the set $\mathrm{Ass}_{R_0}(H^i_{R_{+}}(M)_n)$ (or some
of its specified subsets) ultimately becomes independent of $n$, if $n\rightarrow -\infty$. In many particular
cases this is indeed the case (s. \cite{B1}, \cite{BFL}, \cite{BFT}, \cite{BH}), partly even in a more general setting
(s. \cite{JZ}). On the other hand it is known for quite a while, that the asymptotic stability of associated primes
also may fail in many even surprisingly ``nice'' cases by various examples (s. \cite{BFT}, \cite{BKS} and also \cite{B2}),
which rely on the constructions given in \cite{S} and \cite{SS}.

Another related question is, whether for fixed $i\in \mathbb{N}_0$ certain numerical invariants of the $R_0$-modules
$H^i_{R_{+}}(M)_n$ ultimately become constant if $n\rightarrow -\infty$. A number of such \textit{asymptotic stability results
for numerical invariants} are indeed known (s. \cite{BB}, \cite{BR}, \cite{BRS} and also \cite{HJZ}).

The oldest - and most challenging - question around the asymptotic behaviour of cohomology was the so-called \textit{tameness
problem}, that is the question, whether for fixed $i \in \mathbb{N}_0$ the $R_0$-modules $H^i_{R_{+}}(M)_n$ are either always
vanishing for all $n \ll 0$ or always non-vanishing for all $n \ll 0$. This question seems to have raised already in relation with
Marley's paper \cite{M}. In a number of cases, this tameness problem was shown to have an affirmative answer (s. \cite{B2}, \cite{BH},
\cite{L}, \cite{RS}).

Nevertheless by means of a duality result for bigraded modules given in \cite{HR}, Cutkosky and Herzog \cite{CH} constructed
an example which shows that the tameness-problem can have a negative answer also. In \cite{CCHS} an even more striking counter-example
is given: a Rees-ring $R$ of a three-dimensional local domain $R_0$ of dimension $4$, which is essentially of finite type over a field
such that the graded $R$-module $H^2_{R_{+}}(R)$ is not tame.

The present paper is devoted to the study of the tame loci $\ft^i(M)$ of $M$, that is the sets of all primes $\fp_0
\in \mathrm{Spec}(R_0)$ for which the graded $R_{\fp_0}$-module $H^i_{R_{+}}(M)_{\fp_0} \cong H^i_{{(R_{\fp_0})}_{+}}(M_{\fp_0})$
is tame. This loci have been studied already in \cite{RS}. We restrict ourselves to the case in which the base ring $R_0$ is
essentially of finite type over a field, as in this situation asymptotic stability of associated primes holds in codimension $\leq 2$.
As shown by Chardin-Jouanolou, this latter asymptotic stability result holds under the weaker assumption that $R_0$
is a homomorphic image of a Noetherian ring which is locally Gorenstein (oral communication by M. Chardin). So all results of our
paper remain valid if $R_0$ is subject to this weaker condition.

One expects, that in such a specific situations the tame loci $\ft^i(M)$ show some ``usual'' well-behaviour, like being
open for example. But as we shall see in Example~\ref{2.5} this is wrong in general. Namely, using the counter-example given
in \cite{CCHS} we construct an example of graded $R$-module $M$ of dimension $4$ whose $2$-nd tame locus $\ft^2(M)$ is
not even stable under generalization. This shows in particular, that the tame loci $\ft^i(M)$ need not be open in codimension
$\leq 4$. The example of \cite{CCHS} also shows, that the tame loci $\ft^i(M)$ need not contain all primes
$\fp_0 \in \mathrm{Spec}(R_0)$ of height $3$. Therefore we shall focus to the ``border line case'' and investigate
the sets $\ft^i(M)^{\leq 3}$ of all primes $\fp_0 \in \ft^i(M)$ of height $\leq 3$.

In Section 2 of this paper we recall a few basic facts on the asymptotic stability of associated primes which shall
be used constantly in our arguments. In this section we also introduce the so called \textit{critical sets}
$C^i(M) \subset \mathrm{Spec}(R_0)$ which consist of primes of height $3$ and have the property that all primes $\fp_0 \notin C^i(M)$
of height $\leq 3$ belong to the tame locus $\ft^i(M)$ (s. Proposition~\ref{2.8} (b)). Moreover the finiteness of the set
$C^i(M)$ has the particularly nice consequence that $M$ is \textit{uniformly tame at level $i$ in codimension $\leq 3$}, e.g.
there is an integer $n_0$ such that for each $\fp_0 \in \ft^i(M)^{\leq 3}$ the ${(R_0)}_{\fp_0}$-module
$(H^i_{R_{+}}(M)_n)_{\fp_0}$ is either vanishing for all $n \leq n_0$ or non-vanishing for all $n \leq n_0$
(s. Proposition~\ref{2.8} (c)).

In Section 3 we give some finiteness criteria for the critical sets $C^i(M)$. Here, we assume in addition that the base
ring $R_0$ is a domain, so that the intersection $\fa^i(M)$ of all non-zero primes $\fp_0 \subset R_0$ which are associated
to $H^i_{R_{+}}(M)$ is a non-zero ideal by a result of \cite{BFL}. Our main result says, that the critical set $C^i(M)$ is
finite, if $\fa^i(M)$ contains a \textit{quasi-non-zero divisor} with respect to $M$ (s. Theorem~\ref{3.4}). This obviously
applies in particular to the case in which $M$ is torsion-free as an $R_0$-module in all large degrees or at all
(s. Corollary~\ref{3.5} resp. Corollary~\ref{3.7}). In order to force a situation as required in Theorem~\ref{3.4} one
is tempted to replace $M$ by $M/{\Gamma}_{(x)}(M)$ for some non-zero element $x \in R_0$. We therefore give a
comparison result for the critical sets $C^i(M)$ and $C^i(M/{\Gamma}_{(x)}(M))$ (s. Proposition~\ref{3.7}).
As an application we prove that the critical sets $C^i(M)$ are finite if $R_0$ is a domain and the $R_0$-module $M$
asymptotically satisfies some weak ``unmixedness condition'' (s. Corollary~\ref{3.8}).

In our final Section 4 we give a few conditions for the \textit{tameness at level $i$ in codimension $\leq 3$} in terms
of the ``asymptotic smallness'' of the graded $R$-modules $H^{i-1}_{R_{+}}(M)$ and $H^{i+1}_{R_{+}}(M)$.  We first prove
that all primes $\fp_0 \subset R_0$ of height $\leq 3$ belong to the tame locus $\ft^i(M)$, provided that
$\mathrm{dim}_{R_0}(H^{i-1}_{R_{+}}(M)_n) \leq 1$ and $\mathrm{dim}_{R_0}(H^{i+1}_{R_{+}}(M)_n) \leq 2$ for all $n \ll 0$
(s. Theorem~\ref{4.2}). In addition we show that $M$ is tame at almost all primes $\fp_0 \subset R_0$ of height $\leq 3$
provided that $R_0$ is a domain and  $\mathrm{dim}_{R_0}(H^{i-1}_{R_{+}}(M)_n) \leq 0$ for all $n \ll 0$ (s. Theorem~\ref{4.4}).
We actually prove in both cases slightly sharper statements namely: the corresponding graded $R_{\fp_0}$-modules
$H^i_{R_{+}}(M)_{\fp_0}$ are not only tame, but even what we call \textit{almost Artinian}. Using this terminology
we get in particular the following conclusion. If $R_0$ is a domain and the graded $R$-module $H^{i-1}_{R_{+}}(M)$ is almost
Artinian, then for almost all primes $\fp_0 \in \mathrm{Spec}(R_0)$ of height $\leq 3$ either the ${(R_0)}_{\fp_0}$-module
$(H^i_{R_{+}}(M)_n)_{\fp_0}$ is of dimension $> 0$ for all $n \ll 0$ or else the graded $R_{\fp_0}$-module $H^i_{R_{+}}(M)_{\fp_0}$
is almost Artinian (s. Corollary~\ref{4.5}).

\section{Tame Loci in Codimension $\leq 3$}

 We keep the previously introduced notations.

\begin{con not}\label{2.1}

\rm{(A) Throughout this section we convene that the base ring $R_{0}$
of our Noetherian homogeneous ring $R= R_{0}\bigoplus R_{1}\bigoplus ...$
is essentially of finite type over some field. So, $R_0 = S^{-1}A$, where $A = K[a_1,\ldots,a_s]$ is a finitely generated algebra over some field $K$,
$S \subseteq A$ is multiplicatively closed and there are finitely many elements $l_1,\ldots,l_r \in R_1$ such that $R = R_0[l_1,\ldots,l_r]$.
\\(B) If $n\in\mathbb{N}_{0}$ and $\fP \subseteq \text{Spec}(R_{0})$ we
write}
$$ \fP^{=n} :=\{\fp_{0}\in \fP \mid  \mathrm{height}(\fp_{0})= n\} $$
$$ \fP^{\leq n} :=\{\fp_{0}\in \fP \mid  \mathrm{height}(\fp_{0})\leq n\}. $$
\end{con not}

\begin{rem rem}\label{2.2}

\rm{(A) According to \cite{B0} for all  $n \ll 0$ the set $\mathrm{Ass}_{R_{0}}(M_{n})$ is equal to the set
$\{\fp \cap R_0 \mid \fp \in \mathrm{Ass}_R \cap \mathrm{Proj}(R)\}$ and hence
asymptotically stable for $n\rightarrow \infty$, thus:

\begin{quote} \textit{There is a least integer $m(M) \geq 0$ and a finite set}
$\mathrm{Ass}_{R_{0}}^{*}(M)\subseteq \mathrm{Spec}(R_{0})$  \textit{such that}
$\mathrm{Ass}_{R_{0}}(M_{n})= \mathrm{Ass}_{R_{0}}^{*}(M)$ \textit{for all} $n>
m(M)$.
\end{quote}

(B) Let $f(M)$ denote the finiteness dimension of $M$ with respect
to $R_{+}$, that is "the least integer" for which the $R$-module $
H^{i}_{R_{+}}(M)$ is not finitely generated. Clearly we may write
$$ f(M)= \inf \{i\in \mathbb{N}_{0} \mid \sharp \{ n\in \mathbb{Z} \mid H^{i}_{R_{+}}(M)_{n}\neq 0\}=\infty\}.$$

(C) Keep in mind that $f(M)> 0$. According to [BH, Theorem 5.6] we know that the set
$\mathrm{Ass}_{R_{0}}(H^{f(M)}_{R_{+}}(M)_{n})$ is asymptotically stable for $n\rightarrow -\infty$:
\begin{quote} \textit{There is a largest integer $n(M) \leq 0$ and a finite set} $\fu
(M)\subseteq \mathrm{Spec}(R_{0})$ \textit{such that}
$\mathrm{Ass}_{R_{0}}(H^{f(M)}_{R_{+}}(M)_{n})= \fu (M)$ \textit{for all} $n \leq n(M)$.
\end{quote}

In particular
$$ \mathrm{Supp} _{R_{0}}(H^{f(M)}_{R_{+}}(M)_{n})= \overline{\fu(M)}, \quad \forall n \leq n(M),$$
where $\overline{\bullet}$ denotes the formation of the topological closure in Spec$(R_{0})$.

(D) According to [B1, Theorem 4.1] we know that for each $i\in \mathbb{N}_{0}$ the set\\
$\mathrm{Ass}_{R_{0}}(H^{i}_{R_{+}}(M)_{n})$ is asymptotically stable in codimension $\leq 2$ for $n\rightarrow -\infty$:
\begin{quote} \textit{For each $i\in \mathbb{N}_{0}$ there is a largest integer $n^{i}(M) \leq 0$ and a finite set}
$\fP^{i}(M)\subseteq \mathrm{Spec}(R_{0})^{\leq 2}$  \textit{such that}
$\mathrm{Ass}_{R_{0}}(H^{i}_{R_{+}}(M)_{n})^{\leq 2}= \fP^{i}(M)$ \textit{for all} $n\leq n^{i}(M)$.
\end{quote}
Now, combining this with the observations made in parts (B) and (C) we obtain:}
\begin{quote} {$(i)\, i< f(M) \Rightarrow \forall n\leq n^{i}(M): H^{i}_{R_{+}}(M)_{n}= 0;
\\(ii) \, \forall n\leq n(M): \mathrm{Supp}_{R_{0}}(H^{f(M)}_{R_{+}}(M)_{n})= \overline{\fu (M)};
\\(iii) \, i> f(M)\Rightarrow \forall n\leq n^{i}(M):\mathrm{Supp}_{R_{0}}(H^{i}_{R_{+}}(M)_{n})^{\leq 2}=
\overline{\fP^{i}(M)}^{\leq 2}$.}
\end{quote}
\end{rem rem}

\begin{defn rem}\label{2.3}

\rm{(A) Let $i\in \mathbb{N}_{0}$. We say that the finitely generated graded $R$-module $M$ is
\textit{(cohomologically) tame at level $i$}
if the graded  $R$-module  $H^{i}_{R_{+}}(M)$ is tame, e.g.
$$ \exists n_{0}\in \mathbb{Z} : (\forall n\leq n_{0}: H^{i}_{R_{+}}(M)_{n}= 0)
\vee (\forall n\leq n_{0}: H^{i}_{R_{+}}(M)_{n}\neq 0).$$

(B) Let $\fp_{0}\in \text{Spec}(R_{0})$. We say that $M$ is \textit{(cohomologically) tame at level $i$ in $\fp_{0}$} if
the graded $R_{\fp_{0}}$-module $M_{\fp_{0}}$ is cohomologically tame at level $i$. In view of the graded flat base change
property of local cohomology
it is equivalent to say that the graded $R_{\fp_{0}}$-module $H^{i}_{R_{+}}(M)_{\fp_{0}}$ is tame.

(C) We define the \textit{$i$-th (cohomological) tame locus of $M$} as the set $\ft^{i}(M)$
of all primes $\fp_{0}\in \text{Spec}(R_{0})$ such that  $M$ is
(cohomologically) tame at level $i$ in $\fp_{0}$. So, if $\fp_{0}\in \mathrm{Spec}(R_{0})$ we have
$$\fp_{0}\in \ft^{i}(M)\Leftrightarrow \exists n_{0}\in \mathbb{Z} :  \left\lbrace
\begin{array}{c l}
\forall n\leq n_{0}: \fp_{0}\in \mathrm{Supp}_{R_{0}}(H^{i}_{R_{+}}(M)_{n})\\
or\\
\forall n\leq n_{0}: \fp_{0}\notin \mathrm{Supp}_{R_{0}}(H^{i}_{R_{+}}(M)_{n})
\end{array}
\right.$$

If $k\in \mathbb{N}_{0}$, the set $\ft^{i}(M)^{\leq k}$ is called the
\textit{$i$-th (cohomological) tame locus of $M$ in codimension $\leq k$}.

(D) Let $\fu \subseteq \mathrm{Spec}(R_{0})$. We say that $M$ is \textit{(cohomologically) tame at level $i$ along $\fu$},
if $\fu\subseteq \ft^{i}(M)$. We say that $M$ is \textit{uniformly
(cohomologically) tame at level $i$ along $\fu$} if there is an integer $n_{0}$
such that for all $\fp_{0}\in \fu$
$$\big(\forall n \leq n_0: \fp_{0}\in \mathrm{Supp}_{R_{0}}(H^{i}_{R_{+}}(M)_{n}\big) \vee \big(\forall n \leq n_0: \fp_{0}\notin
\mathrm{Supp}_{R_{0}}(H^{i}_{R_{+}}(M)_{n}\big).$$

(E) If $M$ is uniformly tame at level $i$ along the set $\fu \subseteq \mathrm{Spec}(R_{0})$, then it is tame along $\fu$
at level $i$.}
\end{defn rem}

\begin{rem}\label{2.4}
\rm{(A)} According to Reminder and Remark~\ref{2.2} (D) (i) and (ii) we have
\begin{quote} \textit{$M$ is uniformly tame along} \rm{Spec}$(R_{0})$ \textit{at all levels} $i\leq f(M)$.
\end{quote}

(B) Using the notation of Reminder and Remark~\ref{2.2} (A) we write $\mathrm{Supp}_{R_{0}}^{*}(M):=
\overline{\mathrm{Ass}_{R_{0}}^{*}(M)}$ so that
$\mathrm{Supp}_{R_{0}}(M_{n})= \mathrm{Supp}_{R_{0}}^{*}(M)$  for all $n\geq m(M)$.
Now, on use of Reminder and Remark~\ref{2.2} (D) it follows easily:
\begin{quote} \textit{ For all $i> f(M)$, the module $M$ is uniformly tame at
level $i$ along the set $W^{i}(M):= (\mathrm{Spec}(R_{0})\setminus \mathrm{Supp}_{R_{0}}^{*}(M))
\cup \overline{\fP^{i}(M)}\cup \mathrm{Spec}(R_{0})^{\leq 2}$.}
\end{quote}
It follows in particular that $W^{i}(M)\subseteq \ft^{i}(M)$, and moreover, for all $i\in \mathbb{N}_{0}$:
\begin{quote} \textit{(i) $M$ is uniformly tame at level $i$ along the set $\mathrm{Spec}(R_{0})^{\leq 2}$.
\\ (ii) $\ft^{i}(M)^{\leq 3}$ is stable under generalization.}
\end{quote}
\end{rem}

If the graded $R$-module $T= \bigoplus _{n \in \mathbb{Z}} T_{n}$ is tame,
and $\fp_{0}\in \text{Spec}(R_{0})$, then the graded $R_{\fp_{0}}$-module $T_{\fp_{0}}$ need not to be tame any more.
This hints that in general the loci
$\ft^{i}(M)$ could be non-stable under generalization. We now present such an example.

\begin{exam}\label{2.5}

\rm{Let $K$ be algebraically closed. Then according to [CCHS], there
exists a normal homogeneous Noetherian domain $R'= \bigoplus _{n
\geq 0} R'_{n}$ of dimension 4 such that $(R_{0}', \fm_{0}')$ is
local, of dimension 3 with $R_{0}'/\fm_{0}'= K$ and such
that for all negative integers $n$ we have $H^{2}_{R'_{+}}(R')_{n}=
K^{2}$ if $n$ is even and $H^{2}_{R'_{+}}(R')_{n}= 0$ if $n$ is odd.

Now, let $l_{1}, ..., l_{r}\in R'_{1}$ be such that $R'_{1}= \sum
_{i=1}^{r}R'_{0}l_{i}$. Let $x, x_{1}, ..., x_{r} $ be
indeterminates, let $R_{0}$ denote the 4-dimensional local domain
$R_{0}'[x]_{(\fm_{0}', x)}$ with maximal ideal $\fm_{0}:= (\fm_{0}',
x)R_{0}'$, consider the homogeneous $R_{0}$-algebras $R:=
R_{0}[x_{1}, ..., x_{r}]$ and $\overline{R}:= R_{0}\otimes_{R_{0}'
}R'$ together with the surjective graded homomorphism of
$R_{0}$-algebras
\[\Phi: R= R_{0}[x_{1}, ..., x_{r}]\twoheadrightarrow \overline{R}; \quad x_{i}\mapsto 1_{R_{0}}\otimes l_{i}.\]
Now, let $\alpha \in  \fm_{0}'\backslash \{0\}$, let
$t$ be a further indeterminate, consider the Rees algebra \[S=
R_{0}[xt, (x+ \alpha)t]= \bigoplus_{ n\geq 0}((x,
x+\alpha)R_{0})^{n}\] and the surjective graded homomorphism of
$R_{0}$-algebras
\[\Psi: R\twoheadrightarrow S, \quad x_{1}\mapsto xt, \quad x_{2}\mapsto (x+\alpha)t, \quad x_{i}\mapsto 0  \text{ if $i\geq 3$}.\]
We consider $\overline{R}$ and $S$ as graded $R$-modules by means of
$\Phi$ and $\Psi$ respectively. Then $M:= \overline{R}\oplus S$ is a
finitely generated graded $R$-module which is in addition
torsion-free over $R_{0}$.

By the graded Base Ring Independence and Flat Base Change properties
of local cohomology we get isomorphisms of graded $R$-modules \[
H^{2}_{R_{+}}(\overline{R})\cong R_{0}
\otimes_{R_{0}'}H^{2}_{R_{+}'}(R'), \,\,\,\ H^{2}_{R_{+}}(S)\cong
H^{2}_{S_{+}}(S).\] As $\text{cd}_{S_{+}}(S)=
\text{dim}(S/\fm_{0}S)= 2$
we have $H^{2}_{S_{+}}(S)_{n}\neq 0$ for all $n\ll 0$. It follows that $H^{2}_{R_{+}}(M)_{n}\cong
H^{2}_{R_{+}}(\overline{R})_{n}\oplus  H^{2}_{S_{+}}(S)_{n}\neq 0$
for all $n\ll 0$ and so $M$ is tame at level 2. In particular we
have $\fm_{0}\in \ft^{2}(M)$.

Now, consider the prime $\fp_{0}:= \fm_{0}'R_{0}\in
\mathrm{Spec}(R_{0})^{=3}$. Then, for each $n< 0$ we have
\[(H^{2}_{R_{+}}(\overline{R})_{n})_{\fp_{0}}\cong (R_{0})_{\fm_{0}'R_{0}}\otimes_{R_{0}'} H^{2}_{R_{+}'}(R')_{n}\cong
\left\lbrace
\begin{array}{c l}
  K(x)^{2}, \text{if $n$ is even};\\
  0, \,\,\,\,\,\,\,\,\,\,\, \text{if $n$ is odd}.
\end{array}
\right.\]
Moreover $S_{\fp_{0}}= (R_{0})_{\fp_{0}}[(x,
x+\alpha)(R_{0})_{\fp_{0}}t]= (R_{0})_{\fp_{0}}[t]$ shows that
$H^{2}_{S_{+}}(S)_{\fp_{0}}\cong
H^{2}_{(S_{\fp_{0}})_{+}}(S_{\fp_{0}})= 0$. It follows that
$(H^{2}_{R_{+}}(M)_{n})_{\fp_{0}}$ vanishes precisely for all odd
negative integers $n$. So $H^{2}_{R_{+}}(M)_{\fp_{0}}$ is not tame
and hence $\fp_{0}\notin \ft^{2}(M)$.

Observe in particular that here $\ft^{2}(M)= \ft^{2}(M)^{\leq 4}$ is
not stable under generalization, and that $R_{0}$ is a domain and
the graded $R$-module $M$ is torsion-free over $R_{0}$. On the other
hand $\ft^{i}(M)^{\leq 3}$ is always stable under generalization,
(cf. Remark~\ref{2.4} (B) (ii)).}
\end{exam}

One of our aims is to show that quite a lot can be said about  the
sets $\ft^{i}(M)^{\leq 3}$ if the base ring $R_{0}$ is a domain and
$M$ is torsion-free over $R_{0}$. Indeed, we shall attack the
problem in a more general context, beginning with the following
result, in which $\fP^{i}(M)$ is defined according to Definition and Remark~\ref{2.2} (D).

\begin{lem}\label{2.6}

Let $i\in \mathbb{N}_{0}$ and let $n^i(M)$ be defined as in Reminder and Remark~
\ref{2.2} (D). Then for all $n\leq n^{i}(M)$ we have
\rm{\[C^i_n(M) := \big(\text{Supp}
_{R_{0}}(H^{i}_{R_{+}}(M)_{n})\backslash
\overline{\fP^{i}(M)}\big)^{\leq
3}=\big(\text{Ass}_{R_{0}}(H^{i}_{R_{+}}(M)_{n})\backslash
\overline{\fP^{i}(M)}\big)^{= 3}.\]}
\end{lem}
\begin{proof}
Let $n\leq n^{i}(M)$ and $\fp_{0}\in \big((\mathrm{Supp}_{R_{0}}(H^{i}_{R_{+}}(M)_{n})\backslash \overline{\fP^{i}(M)}\big)^{\leq 3}$.
Then, there is some $\fq_{0}\in \mathrm{Ass}_{R_{0}}(H^{i}_{R_{+}}(M)_{n})$ with $\fq_{0}\subseteq \fp_{0}$.
As $\fp_{0}\notin \overline{\fP^{i}(M)}$ we have $\fq_{0}\notin\fP^{i}(M)= \mathrm{Ass}_{R_{0}}(H^{i}_{R_{+}}(M)_{n})^{\leq 2}$.
It follows that $\mathrm{height} (\fq_{0}) \geq 3$, hence $\fq_{0}= \fp_{0}$ and therefore
$$\fp_{0}\in \mathrm{Ass}_{R_{0}}(H^{i}_{R_{+}}(M)_{n})^{=3}.$$
This proves the inclusion $"\subseteq"$. The converse inclusion is obvious.
\end{proof}

\begin{defn}\label{2.7}

\rm{Let $i \in \mathbb{N}_{0}$ and let $n^i(M)$ and $C^i_n(M)$ be as in Lemma~\ref{2.6}. Then the set
$$ C^i(M) := \bigcup_{n \leq n^i(M)} C^i_n(M)$$
is called the \textit{$i$-th critical set of $M$}}.
\end{defn}

\begin{prop}\label{2.8}
Let $i \in \mathbb N_{0}$. Then
\\{\rm(a)} $M$ is uniformly tame at level $i$ along the set
\[ [\big(\mathrm{Spec}(R_{0})\setminus \mathrm{Supp}^{*}_{R_{0}}(M)\big)\cup
\overline{\fP^{i}(M)}\cup \mathrm{Spec}(R_{0})^{\leq 3}]\setminus
C^{i}(M).\]
\\{\rm(b)} $\ft^{i}(M)^{\leq 3}\supseteq \mathrm{Spec}(R_{0})^{\leq 3}\setminus C^{i}(M).$
\\{\rm(c)} The following statements are equivalent:
\\{\rm(i)} $C^{i}(M)$ is a finite set;
\\{\rm(ii)} $\ft^{i}(M)^{\leq 3}$ is open in
$\mathrm{Spec}(R_{0})^{\leq 3}$ and $M$ is uniformly tame at level $i$ along $\ft^{i}(M)^{\leq 3}$.
\\{\rm(iii)} $\mathrm{Spec}(R_{0})^{\leq 3} \setminus \ft^{i}(M)$ is finite and $M$ is uniformly tame at
level $i$ along $\ft^{i}(M)^{\leq 3}$.
\end{prop}

\begin{proof}
(a): This follows from Remark~\ref{2.4} (B) and the fact that
$$\big[\bigcup_{n\leq n^{i}(M)}\mathrm{Supp}_{R_{0}}(H^{i}_{R_{+}}(M)_{n})\big]^{=3}\setminus \overline{\fP^{i}(M)}= C^{i}(M).$$
\\(b): This is immediate by statement (a).
\\(c): "(i) $\Rightarrow$ (ii)": This follows easily by statements (a) and (b)
and the fact that $M$ is uniformly tame at level $i$ along each finite subset $V\subseteq \fP^{i}(M)$.
\\"(ii) $\Rightarrow$ (iii)``: Assume that statement (ii) holds. As $\mathrm{Spec}(R_{0})^{\leq 2}\subseteq \ft^{i}(M)^{\leq 3}$
(s. Remark~\ref{2.4} (B) (i)) and as $\ft^{i}(M)^{\leq 3}$ is open in $\mathrm{Spec}(R_{0})^{\leq 3}$
it follows that $\mathrm{Spec}(R_{0})^{\leq 3}\setminus \ft^{i}(M)^{\leq 3}$ is a finite set, and this proves statement (iii).
\\"(iii) $\Rightarrow$ (i)'': Assume that statement (iii) holds so that $\mathrm{Spec}(R_0)^{\leq 3}\setminus \ft^i(M)$ is finite
and $M$ is uniformly tame along $\ft^i(M)^{\leq 3}$. By statement (b) we have
$\mathrm{Spec}(R_{0})^{\leq 3}\setminus \ft^{i}(M)^{\leq 3} \subseteq C^{i}(M)\subseteq\mathrm{Spec}(R_{0})^{= 3}$.
It thus suffices to show that the set $F:= C^{i}(M)\cap \ft^{i}(M)$ is finite.

By uniform tameness there is some integer $n_{0}\leq n^{i}(M)$ such that for each $\fp_{0}\in F$ either
\[(I)\,\ \fp_{0}\in \mathrm{Supp}_{R_{0}}(H^{i}_{R_{+}}(M)_{n})\,\ \textrm{for all}\,\ n\leq n_{0}; \mathrm{or}\]
\[(II)\,\ \fp_{0}\notin \mathrm{Supp}_{R_{0}}(H^{i}_{R_{+}}(M)_{n})\,\ \textrm{for all}\,\ n\leq n_{0}.\]

Let $F_{I}:= \{\fp_{0}\in F \mid \fp_{0}\, \text{satisfies}\, (I)\}$ and
$F_{II}:= \{\fp_{0}\in F \mid \fp_{0} \, \text{satisfies}\, (II)\}$.
As $F= F_{I}\cup F_{II}$ it suffices to show that $F_{I}$ and $F_{II}$ are finite.

If $\fp_{0}\in F_{I}$,
we have $\fp_{0}\in \big(\mathrm{Supp}_{R_{0}}(H^{i}_{R_{+}}(M)_{n_{0}})\setminus \overline{\fP^{i}(M)}\big)^{\leq 3}$.
As $n_{0}\leq n^{i}(M)$ statement (a) implies $\fp_{0}\in \mathrm{Ass}_{R_{0}}(H^{i}_{R_{+}}(M)_{n_{0}})$.
This proves that $F_{I}\subseteq  \mathrm{Ass}_{R_{0}}(H^{i}_{R_{+}}(M)_{n_{0}})$ and  thus $F_{I}$ is finite.

Clearly
$F_{II}\subseteq  \big(\bigcup_{n_{0}\leq n\leq n^{i}(M)} \mathrm{Supp}_{R_{0}}(H^{i}_{R_{+}}(M)_{n}\setminus
\overline{\fP^{i}(M)}\big)^{\leq 3}$.
So, by statement (a) we see that $F_{II}$ is contained in the finite set
$\bigcup_{n_{0}\leq n\leq n^{i}(M)}\mathrm{Ass}_{R_{0}}(H^{i}_{R_{+}}(M)_{n})$.
\end{proof}

 \section{Finiteness of Critical sets}

We keep all notations and hypotheses of the previous section. So $R = \bigoplus _{n \in \mathbb{N}_0} R_n$ is
a Noetherian homogeneous ring whose base ring $R_0$ is essentially of finite type over some field and $M$ is a
finitely generated graded $R$-module. By statement (c) of Proposition 2.8 it seems quite appealing to look
for criteria which ensure that the critical sets $C^{i}(M)$ are
finite. This is precisely the aim of the present section.

\begin{remi}\label{3.1}

\rm{ (A) Assume that $R_{0}$ is a domain. Then, according to [BFL, Theorem
2.5] there is an element $s\in R_{0}\backslash \{0\}$ such that the
$(R_{0})_{s}$-module $(H^{i}_{R_{+}}(M))_{s}$ is torsion-free
or 0 for all $i\in \mathbb{N}_{0}$. From this we conclude that
(with the standard convention that $\bigcap_{\fp_{0}\in \emptyset}\fp_{0}:=R_{0}$):}

\begin{quote} \textit{If $R_{0}$ is a domain, the ideal \[ \fa^{i}(M):=
\bigcap_{\fp_{0}\in \mathrm{Ass}_{R_{0}}(H^{i}_{R_{+}}(M))\setminus
\{0\}}\fp_{0} \]
is $\neq 0$ for all $i\in \mathbb{N}_{0}$.}
\end{quote}

(B) Keep the notations and hypotheses of part (A). Then:
\begin{quote} \textit{If $x \in \fa^i(M)$ and if $N$ is a second finitely generated graded $R$-module such that the graded $R_x$-modules
$M_x$ and $N_x$ are isomorphic, then $x \in \fa^i(N)$.}
\end{quote}
This follows immediately from the fact, that for all $n \in \mathbb{Z}$ there is an isomorphism of $(R_0)_x$-modules
$(H^i_{R_{+}}(M)_n)_x \cong (H^i_{R_{+}}(N)_n)_x$.
For our purposes the most significant application of this observation is:
\begin{quote} \textit{If $x \in \fa^i(M)$ then $x \in \fa^i(M/{\Gamma}_{(x)}(M))$.}
\end{quote}
\end{remi}

\begin{nota}\label{3.2}

\rm{An element $x\in R_{0}$ is called a \textit{quasi-non-zero divisor with respect
to (the finitely generated graded $R$-module) $M$} if $x$ is a non-zero divisor
on $M_{n}$ for all $n\gg 0$. We denote the set of these
quasi-non-zero divisors by $\mathrm{NZD}^{*}_{R_{0}}(M)$.
Thus in the notation of Reminder and Remark~\ref{2.2} (A) we may write}
$$\mathrm{NZD}^{*}_{R_{0}}(M)= R_{0}\backslash \bigcup _{\fp_{0}\in \mathrm{Ass}_{R_{0}}^{*}(M)}\fp_{0}$$
\end{nota}

\begin{lem}\label{3.3}
Let $i, k\in \mathbb{N}_{0}$ and assume that $\rm{height}$$(\fp_{0})\geq k$ for all $\fp_{0}\in\rm{Ass}$ $_{R_{0}}^{*}(M).$
Then, the set $ \rm{Ass}$$_{R_{0}}(H^{i}_{R_{+}}(M)_{n})^{\leq k+2}$ is asymptotically stable for $n\rightarrow -\infty $.
In particular, if $k> 0$, then $C^{i}(M)$ is finite.
\end{lem}

\begin{proof}
There is some integer $n_{0}\in \mathbb{Z}$ such that $(0:_{R_{0}}M_{\geq n_{0}})\subseteq R_{0}$ is of height $\geq k$,
where we use the notation $M_{\geq n_{0}}:= \bigoplus _{n \geq n_{0}}M_{n}$.
As $H^{i}_{R_{+}}(M)$ and $H^{i}_{R_{+}}(M_{\geq n_{0}})$ differ only in finitely many degrees we may replace
$M$ by $M_{\geq n_{0}}$ and hence assume that $\fa_{0}M= 0$ for some ideal $\fa_{0}\subseteq R_{0}$ with height$(\fa_{0})\geq k$.
As height$(\fp_{0}/\fa_{0})\leq $ height$(\fp_{0})- k$ for all $\fp_{0}\in \rm{Var}(\fa_{0})$
and in view of the natural isomorphisms of $R_{0}$-modules $H^{i}_{R_{+}}(M)_{n}
\cong H^{i}_{(R/\fa_{0}R)_{+}}(M)_{n}$ we now get
a canonical bijection
\[\mathrm{Ass}_{R_{0}}(H^{i}_{R_{+}}(M)_{n})^{\leq k+ 2}\leftrightarrow
\mathrm{Ass}_{R_{0}/\fa_{0}}(H^{i}_{R_{+}}(M)_{n})^{\leq 2},\]
for all $n\in \mathbb{Z}$. So, by Reminder and Remark~\ref{2.2} (D) the left hand side set is asymptotically stable for
$n\rightarrow -\infty $.
If $k> 0$ the finiteness of $C^{i}(M)$ now follows easily from statement (a) of Lemma 2.6.
\end{proof}

Let $i \in \mathbb{N}_0$. According to Remark~\ref{2.4} (B) we know that $M$ is uniformly tame at level $i$ in codimension $\leq 2$.
we also know that $M$ need not be tame at level $i$ in codimension $3$. It is natural to ask, whether there are only finitely
many primes $\fp_{0}$ of height $3$ in $R_0$ such that $M$ is not tame at level $i$ in $\fp_{0}$ and whether outside of
these ``bad'' primes the module $M$ is uniformly tame at level $i$ in codimension $\leq 3$. We aim to give a few
sufficient criteria for this behaviour. The following proposition plays a crucial r\^ole in this respect.

\begin{thm}\label{3.4}
Let $i\in \mathbb{N}_{0}$. Assume that $R_{0}$ is a domain and that
$\mathrm{NZD}^{*}_{R_{0}}(M)\cap \fa^{i}(M)\neq \emptyset$. Then $C^{i}(M)$
is a finite set. In particular the set $\mathrm{Spec}(R_{0})^{\leq 3} \setminus \ft^i(M)$ consists of finitely many primes
of height $3$ and $M$ is uniformly tame at level $i$ along $\ft^{i}(M)^{\leq 3}$.
\end{thm}

\begin{proof}
If $i\leq f(M)$ our claim is clear by Remark~\ref{2.4} (A) and Proposition~\ref{2.8} (c). So, let $i> f(M)$. Then in particular $i> 1$.

Now, let $m(M)\in \mathbb{Z}$ be as in Reminder and Remark~\ref{2.2} (A) and set $N:= M_{\geq m(M)}:= \bigoplus _{n \geq m(M)} M_{n}$.
Then NZD$^{*}_{R_{0}}(M)$ equals the set NZD$_{R_{0}}(N)$ of non-zero divisors in $R_{0}$ on $N$.
As $i> 1$ we have $H^{i}_{R_{+}}(N)= H^{i}_{R_{+}}(M)$ and hence  $\fa^{i}(M)= \fa^{i}(N)$
and $C^{i}(M)= C^{i}(N)$. So, we may replace $M$ by $N$ and hence assume that NZD$_{R_{0}}(M)\cap
\fa^{i}(M)\neq \emptyset$.

Let $x\in  \mathrm{NZD}_{R_{0}}(M)\cap \fa^{i}(M)$. Then, the short exact sequence
$0\longrightarrow M \stackrel{x}{\longrightarrow }M \longrightarrow M/xM \longrightarrow 0$ implies exact
sequences
\[H^{i}_{R_{+}}(M)_{n}\stackrel{x}{
\longrightarrow}H^{i}_{R_{+}}(M)_{n}\longrightarrow
H^{i}_{R_{+}}(M/xM)_{n}\]
for all $n\in \mathbb{Z}$. Now, let $\fp_{0}\in C^{i}(M)$ so that $\mathrm{height}(\fp_0 )= 3$ (s. Lemma~\ref{2.6}).
Then, there is an integer $n\leq n^{i}(M)$
such that $\fp_{0}$ is a minimal associated prime of $H^{i}_{R_{+}}(M)_{n}$.
We thus get an exact sequence of $(R_{0})_{\fp_{0}}$-modules
\[(H^{i}_{R_{+}}(M)_{n})_{\fp_{0}}\stackrel{\frac{x}{1}}{\longrightarrow}
(H^{i}_{R_{+}}(M)_{n})_{\fp_{0}}\stackrel{\varrho}{\longrightarrow}(H^{i}_{R_{+}}(M/xM)_{n})_{\fp_{0}}\]
in which the middle module is of finite length $\neq 0$.
As $x\in \fa^{i}(M)\subseteq \fp_{0}$
it follows by Nakayama that $\varrho$ is not the zero map. Therefore $(H^{i}_{R_{+}}(M/xM)_{n})_{\fp_{0}}$
contains a non-zero $(R_{0})_{\fp_{0}}$-module of finite length. It follows that
$\fp_{0}\in \mathrm{Ass}_{R_{0}}(H^{i}_{R_{+}}(M/xM)_{n}) ^{= 3}$.
This shows that $C^{i}(M)\subseteq \text{Ass}_{R_{0}}(H^{i}_{R_{+}}(M/xM)_{n}) ^{= 3}$.
So, by Lemma~\ref{3.3} the set
$C^{i}(M)$ is finite.
\end{proof}

\begin{cor}\label{3.5}

Let $i\in \mathbb{N}_{0}$. Assume that $R_{0}$ is a domain and that
$M_{n}$ is a torsion-free $R_{0}$-module
for all $n\gg 0$. Then the set $C^{i}(M)$ is finite. In particular, $M$ is uniformly tame at level $i$ along
$\ft^i(M)^{\leq 3}$ and the set $\rm{Spec}(R_0)^{\leq 3} \setminus \ft^i(M)$ is finite.
\end{cor}

\begin{proof}
By our hypotheses we have $\mathrm{NZD}^{*}_{R_{0}}(M)= R_{0}\setminus \{0\}$.
By Reminder~\ref{3.1} (A) we have $\fa^{i}(M)\neq 0$.
Now we conclude by Theorem~\ref{3.4}.
\end{proof}

\begin{cor}\label{3.6}

Let $i \in \mathbb{N}_0$ and assume that $R_0$ is a domain and $M$ is torsion-free over $R_0$.
Then $M$ is uniformly tame at level $i$ along a set which is obtained by removing finitely many primes
of height $3$ from $\mathrm{Spec}(R_0)^{\leq 3}$.
\end{cor}

\begin{proof} This is clear by Corollary~\ref{3.5}.
\end{proof}

Our next aim is to replace the requirement that $M_n$ is $R_0$ torsion-free for all $n \gg 0$, which was used
in Corollary~\ref{3.5} by a weaker condition. We begin with the following finiteness result for certain subsets of critical sets:

\begin{prop}\label{3.7}
Let $R_{0}$ be a domain, let $i\in \mathbb{N}$ and let $x\in R_{0}\setminus \{0\}$ be such
that $x\it{\Gamma}_{(x)}(M)= 0$. Then
\\{\rm(a)} $[C^{i}(M)\setminus \large[ C^{i}(M/ \it{\Gamma}_{(x)}(M))\cup
[\overline{\fP^{i- 1}(M/xM)}\cap \overline{\fP^{i+ 1}(\it{\Gamma}_{(x)}(M))}]^{= \mathrm{3}}\large]$
is a finite set.
\\{\rm(b)} If $x \in \fa^i(M)$, then the set $C^i(M/{\Gamma}_{(x)}(M))$ and hence also the set
$$C^i(M)\setminus \big[[\overline{\fP^{i-1}(M/xM)}\cap \overline{\fP^{i+1}({\Gamma}_{(x)}(M))}]^{=3}
\setminus C^i(M/{\Gamma}_{(x)}(M))\big]$$
is finite.
\end{prop}

\begin{proof}
(a): Fix an integer $n_{0} \leq n^{i}(M/xM), n^{i}(\it{\Gamma}_{(x)}(M)),
n^{i}(M), n^{i}(M/ \it{\Gamma}_{(x)}(M))$ and let $\fp_{0}\in C^{i}(M)$.
Then $\fp_{0}\in \text{min Ass}_{R_{0}}(H^{i}_{R_{+}}(M)_{n})$ for
some $n\leq n^{i}(M)$. If $n_{0}\leq n$, $\fp_{0}$ thus belongs to
the finite set $\bigcup_{m\geq
n_{0}}\text{Ass}_{R_{0}}(H^{i}_{R_{+}}(M)_{m})$. So, let $n< n_{0}$.
The graded short exact sequences
\[0\longrightarrow M/{\Gamma}_{(x)}(M)\longrightarrow M\longrightarrow M/xM\longrightarrow 0\]
and
\[0\longrightarrow {\Gamma}_{(x)}(M)\longrightarrow M\longrightarrow M/ {\Gamma}_{(x)}(M)\longrightarrow 0\] imply
exact sequences
\[(H^{i-1}_{R_{+}}(M/xM)_{n})_{\fp_{0}}\longrightarrow
(H^{i}_{R_{+}}(M/{\Gamma}_{(x)}(M))_{n})_{\fp_0} \longrightarrow
(H^{i}_{R_{+}}(M)_{n})_{\fp_0}\longrightarrow
(H^{i}_{R_{+}}(M/xM)_n)_{\fp_0}\] and
\[(H^{i}_{R_{+}}(M)_{n})_{\fp_0}\longrightarrow (H^{i}_{R_{+}}(M/ \it{\Gamma}_{(x)}(M))_{n})_{\fp_0}
\longrightarrow (H^{i+1}_{R_{+}}({\Gamma}_{(x)}(M))_{n})_{\fp_0}.\]
Assume that $\fp_{0}\notin C^{i}(M/{\Gamma}_{(x)}(M))$. Then
$(H^{i}_{R_{+}}(M/{\Gamma}_{(x)}(M))_{n})_{\fp_{0}}$ either vanishes
or is an $(R_{0})_{\fp_{0}}$-module of infinite length. In the first
case we have $(H^{i}_{R_{+}}(M)_{n})_{\fp_{0}}\subseteq (H^{i}_{R_{+}}(M/xM)_{n})_{\fp_{0}}$. As
$(H^{i}_{R_{+}}(M)_{n})_{\fp_{0}}$ is a non-zero
$(R_{0})_{\fp_{0}}$-module of finite length it follows $\fp_{0}\in
\mathrm{Ass}_{R_{0}} (H^{i}_{R_{+}}(M/xM)_{n})$. So $\fp_{0}$ belongs
to the finite set $\mathrm{Ass}_{R_{0}}(H^{i}_{R_{+}}(M/xM))^{\leq 3}$ (s. Remark 3.3).

Assume now that $(H^{i}_{R_{+}}(M/ \it{\Gamma}_{(x)}(M))_{n})_{\fp_{0}}$
is not of finite length. Then, by the above sequences
$(H^{i-1}_{R_{+}}(M/xM)_{n})_{\fp_{0}}$ and $(H^{i+1}_{R_{+}}(\it{\Gamma}_{(x)}(M))_{n})_{\fp_{\mathrm{0}}}$
are both of infinite length, so that $\fp_{0}\in \overline{\fP^{i- 1}(M/xM)}$
and $\fp_{0}\in \overline{\fP^{i+ 1}(\it{\Gamma}_{(x)}(M))}$.

(b): According to Reminder~\ref{3.1} (B) we have $x \in \fa^i(M/{\Gamma}_{(x)}(M))$. As moreover it holds
$x \in \mathrm{NZD}_{R_0}(M/{\Gamma}_{(x)}(M))$ our claim follows be Theorem~\ref{3.4}.
\end{proof}

\begin{cor}\label{3.8}

Let $i\in \mathbb{N}_{0}$, let $R_{0}$ be a domain and assume that $\mathrm{height}(\fp_{0})\geq 3$ \textit{for all}
$\fp_{0}\in  \mathrm{Ass}_{R_{0}}^{*}(M)\backslash \big(\{0\}\cup \overline{\fP^{i}(M)}\big)$. \textit{Then $C^{i}(M)$ is a
finite set. In particular the set $\mathrm{Spec}(R_0)^{\leq 3} \setminus \ft^i(M)$ is finite and $M$ is uniformly tame at
level $i$ along the set $\ft^i(M)^{\leq 3}$.}
\end{cor}

\begin{proof}
Let $m(M)\in \mathbb{Z}$ be as in Reminder and Remark~\ref{2.2} (A) so that $\mathrm{Ass}_{R_{0}}(M_{n})=
\mathrm{Ass}_{R_{0}}^{*}(M)$ for all $n\geq m(M)$.
As $H^{i}_{R_{+}}(M)$ and $H^{i}_{R_{+}}(M_{\geq m(M)})$ differ only in finitely many degrees we may replace $M$ by $M_{\geq m(M)}$
and hence assume that $\mathrm{Ass}_{R_{0}}^{*}(M)= \mathrm{Ass}_{R_{0}}(M)$. If $0\notin \mathrm{Ass}_{R_{0}}(M)$
we get our claim by Lemma~\ref{3.3}.
So, let $0\in  \mathrm{Ass}_{R_{0}}(M)$ and consider the non-zero ideal
$\fb_{0}:= \bigcap _{\fp_{0}\in \mathrm{Ass}_{R_{0}}(M)\setminus \{0\}}\fp_{0}$.
Then $\mathrm{Ass}_{R_{0}}(M/{\Gamma}_{\fb_{0}}(M))= \{0\} $ so that $M/{\Gamma}_{\fb_{0}}(M)$
is torsion-free over $R_{0}$. Let $x\in \fb_{0}\setminus \{0\}$ with $x{\Gamma}_{(x)}(M) = 0$. Then it follows that
${\Gamma}_{\fb_{0}}(M)= {\Gamma}_{(x)}(M)$. By Corollary~\ref{3.5} we therefore obtain that
$C^{i}(M/{\Gamma}_{(x)}(M))$ is finite. According to Proposition~\ref{3.7} (a) it thus suffices to show that
$C^{i}(M)\cap \overline{\fP^{i + 1}({\Gamma}_{\fb_{0}}(M))}^{= 3}$ is finite. So, let $\fq_{0}$ be an element of this latter
set. Then
$\mathrm{height}(\fq_{0})= 3$ and $\fq_{0}\notin \overline{\fP^{i}(M)}$. Moreover, there is a minimal prime $\fp_{0}$ of $\fb_{0}$ with
$\fp_{0}\subseteq \fq_{0}$. In particular $\fp_{0}\in \mathrm{Ass}_{R_{0}}(M)\setminus \{0\}$ and $\fp_{0}\notin \overline{\fP^{i}(M)}$.
So, by our hypothesis height$(\fp_{0})\geq 3$, whence
$\fq_{0}= \fp_{0}\in \mathrm{Ass}_{R_{0}}^{*}(M)\setminus \{0\}$. This shows that
$\overline{C^{i}(M)\cap \fP^{i + 1}({\Gamma}_{\fb_{0}}(M))}^{= 3}\subseteq \mathrm{Ass}_{R_{0}}^{*}(M)$ and hence proves our claim.
\end{proof}

\begin{rem}\label{3.9}
\rm{ Clearly Corollary~\ref{3.6} applies to the domain $R'$ constructed in \cite{CCHS} (s. Example~\ref{2.5}),
taken as a module over itself. In this example we have in particular $\ft^2(R')^{\leq 3} = \mathrm{Spec}(R'_{0}) \setminus \{\fm_0\}$.
Moreover the uniform tameness of $R'$ at level $2$ along this set can be verified by a direct calculation.}
\end{rem}

\section{Conditions on Neighbouring Cohomologies for Tameness in Codimensions $\leq 3$}

We keep the hypotheses and notations of the previous sections. So $R = \bigoplus_{n \in \mathbb{N}_0}R_n$ is a homogeneous
Noetherian ring whose base ring $R_0$ is essentially of finite type over a field and $M$ is a finitely generated graded $R$-module.

Our first result says that $M$ is tame in codimension $\leq 3$ at a given level $i \in \mathbb{N}$, if the two
neigbouring local cohomology modules $H^{i-1}_{R_{+}}(M)$ and $H^{i+1}_{R_{+}}(M)$ are ``asymptotically sufficiently small''.
We actually shall prove a more specific statement. To formulate it, we first introduce an appropriate notion.

\begin{defn rem}\label{4.1}
\rm{ (A) We say that a graded $R$-module $T = \bigoplus_{n \in \mathbb{Z}} T_n$ is \textit{almost Artinian} if there is
some graded submodule $N = \bigoplus_{n \in \mathbb{Z}} N_n \subseteq T$ such that $N_n = 0$ for all $n\ll 0$ and such
that the graded $R$-module $T/N$ is Artinian.

(B) A graded $R$-module $T$ which is the sum of an Artinian graded submodule and a Noetherian graded submodule clearly is
almost Artinian. Moreover, the property of being almost Artinian passes over to graded subquotients.

(C) As $R_0$ is Noetherian and $R$ is homogeneous each graded almost Artinian $R$-module $T$ has the property that
$\mathrm{dim}_{R_0}(T_n) \leq 0$ for all $n\ll 0$.

(D) Clearly an almost Artinian graded $R$-module is tame.}
\end{defn rem}

Now, we are ready to formulate and to prove the announced result.

\begin{thm}\label{4.2}
Let $i\in \mathbb{N}$ such that $\dim_{R_0}(H^{i-1}_{R_+}(M)_n)\leq 1$
and $dim_{R_0}(H^{i-2}_{R_+}(M)_n)\leq 2$ for all $n \ll 0$. Then the following statements hold.\\
{\rm(a)} The graded $R_{\fp_0}$-module $H^i_{R_{+}}(M)_{\fp_0}$ is almost Artinian for all
$\fp_0 \in \mathrm{Spec}(R_0)^{= 3} \setminus \overline{\fP^i(M)}$.
{\rm(b)} $\ft^i(M)^{\leq 3} = \mathrm{Spec}(R_0)^{\leq 3}$ and hence $M$ is tame at level $i$ in codimension $\leq 3$.
\end{thm}
\begin{proof}

(a): Let $\fp_0\in \textrm{Spec}(R_0)^{=3} \setminus \overline{\fP^i(M)}$. We consider the Grothendieck spectral sequence
$$E_2^{p, q}= H^{p}_{\fp_0}(H^{q}_{R_+}(M))_{\fp_0}\underset{p}{\Rightarrow}H^{p+q}_{\fp_0+R_+}(M)_{\fp_0}.$$
By our assumption on the dimension of the $R_0$-modules $H^{i-1}_{R_{+}}(M)_n$ and $H^{i+1}_{R_{+}}(M)_n$, the $n$-th graded
component $(E_2^{p, q})_n$ of the graded $R_{{\fp}_0}$-module $E_2^{p, q}$ vanishes for all $n \ll 0$ if $(p, q)= (2, i-1)$
or $(p, q)= (3, i-2)$. Therefore
$$(E_2^{0, i})_n \cong (E_{\infty}^{0, i})_n, \quad \forall n\ll 0.$$
As the graded $R_{{\fp}_0}$-module $E_{\infty}^{0,i}$ is a subquotient of the Artinian $R_{\fp_0}$-module $H^{i}_{\fp_0+R_+}(M)_{\fp_0}$,
it follows by Definition and Remark~\ref{4.1} (B) that the graded $R_{\fp_0}$-module
$$H^0_{{{\fp}_0} R_{{\fp}_0}}\big(H^i_{R_{+}}(M)_{{\fp}_0}\big) \cong H^{0}_{\fp_0}(H^{i}_{R_+}(M))_{\fp_0} = E_{2}^{0,i}$$
is almost Artinian. Now, since $\fp_0\notin \overline{\fP^i(M)}$ and ${\fp}_0$ is of height $3$ we must have
$$\dim_{{R_0}_{{\fp}_0}}\big((H^i_{R_{+}}(M)_{{\fp}_0})_n\big) \leq 0, \quad \forall n\ll 0.$$
and hence  $H^0_{{{\fp}_0} R_{{\fp}_0}}\big(H^i_{R_{+}}(M)_{{\fp}_0}\big)$ and $H^{i}_{R_+}(M)_{\fp_0}$ coincide in all
degrees $n \ll 0$. Therefore $H^i_{R_{+}}(M)_{\fp_0}$ is indeed almost Artinian.

(b): This follows immediately from statement (a), as $\overline{\fP^i(M)} \subseteq \ft^i(M)$ (s. Remark~\ref{2.4} (B)).
\end{proof}

\begin{rem}\label{4.3} {\rm The domain $R'$ constructed in \cite{CCHS} (s. Example~\ref{2.5}), taken as a module over itself,
clearly cannot satisfy the hypotheses of Theorem~\ref{4.1} with $i = 2$ as it does not fulfill the corresponding
conclusion of this theorem. Indeed a direct calculation shows that $\text{dim}_{R'_0}(H^1_{R'_{+}}(R')_n) = 3$ for all $n < 0$.}
\end{rem}

Our next result says that the module $M$ is tame at level $i$ almost everywhere in codimension $\leq 3$ provided that $R_0$ is
a domain and the local cohomology module $H^{i-1}_{R_{+}}(M)$ is ``asymptotically very small``. Again, we aim to prove a more
specific result.

\begin{thm}\label{4.4}
Let $R_0$ be a domain and $i\in \mathbb{N}$ such that $\mathrm{dim}_{R_0}(H^{i-1}_{R_+}(M)) \leq 0$ for all $n \ll 0$.
Then the following statements hold.\\
{\rm(a)} There is a finite set $Z \subset \mathrm{Spec}(R_0)^{= 3}$ such that the graded $R_{\fp_0}$-module $H^i_{R_{+}}(M)_{\fp_0}$
is almost Artinian for all $\fp_0 \in \mathrm{Spec}(R_0)^{=3} \setminus \big(Z \cup \overline{\fP^i(M)}\big)$.\\
{\rm(b)} $\mathrm{Spec}(R_0)^{\leq 3}\setminus \ft^i(M)$ is a finite subset of $\mathrm{Spec}(R_0)^{= 3}$.
\end{thm}
\begin{proof}
(a): According to Reminder~\ref{3.1} (A) there is an element $x\in \fa^i(M) \setminus \{0\}$ such that $x{\Gamma}_{(x)}(M) = 0$.
If we apply Lemma~\ref{3.3} with $k=1$ to the the $R$-module $M/xM$ (also with $i-1$ instead of $i$)
and to the $R$-module ${\Gamma}_{(x)}(M)$ (with $i+1$ instead of $i$) we see that the three sets
$$\mathrm{Ass}_{R_0}(H^{i-1}_{R_{+}}(M/xM)_n)^{\leq 3}, \quad \mathrm{Ass}_{R_0}(H^i_{R_{+}}(M/xM)_n)^{\leq 3}, \quad
\mathrm{Ass}_{R_0}(H^i_{R_{+}}({\Gamma}_{(x)}(M)_n)^{\leq 3}$$
are asymptotically stable for $n \rightarrow - \infty$. So, there is a finite set $Z\subset \mathrm{Spec}(R_0)^{= 3}$ such that
$$\mathrm{Ass}_{R_0}(H^{i-1}_{R_{+}}(M/xM)_n)^{= 3} \cup \mathrm{Ass}_{R_0}(H^i_{R_{+}}(M/xM)_n)^{= 3} \cup
\mathrm{Ass}_{R_0}(H^{i+1}({\Gamma}_{(x)}(M)_n)^{=3} = Z$$
for all $n\ll 0$.
Let
$$\fp_0 \in \mathrm{Spec}(R_0)^{= 3} \setminus \big(Z \cup \overline{\fP^i(M)}\big).$$
We aim to show that the graded $R_{\fp_0}$-module $H^i_{R_{+}}(M)_{\fp_0}$ is almost Artinian. As
$\fp_0 \notin \overline{\fP^i(M)}$ and $\mathrm{height}(\fp_0) = 3$
it follows

$$ \mathrm{lenght}_{(R_0)_{\fp_0}}(H^i_{R_{+}}(M)_n)_{\fp_0}) < \infty$$

for all $n \ll 0$. As $\mathrm{dim}_{R_0}(H^{i-1}_{R_{+}}(M)_n) \leq 0$ for all $n \ll 0$ we also have

$$ \mathrm{length}_{(R_0)_{\fp_0}}(H^{i-1}_{R_{+}}(M)_n)_{\fp_0} < \infty$$

for all $n \ll 0$. As $\fp_0 \notin Z$ and $\mathrm{height}(\fp_0) = 3$, we also can say

$$ {\Gamma}_{\fp_0(R_0)_{\fp_0}}\big((H^{i-1}_{R_{+}}(M/xM)_n)_{\fp_0}\big) =
{\Gamma}_{\fp_0 (R_0)_{\fp_0}}\big((H^i_{R_{+}}(M/xM)_n)_{\fp_0}\big) =$$
$$={\Gamma}_{\fp_0(R_0)_{\fp_0}}\big((H^{i+1}_{R_{+}}({\Gamma}_{(x)}(M))_n)_{\fp_0}\big) = 0, \quad \forall n\ll 0.$$
Now, as in the proof of Proposition~\ref{3.8} (a), the canonical graded short exact sequences
$$0\longrightarrow M/{\Gamma}{(x)}(M)\stackrel{\phi}{\longrightarrow} M\longrightarrow M/xM \longrightarrow 0$$
and
$$0\longrightarrow {\Gamma}_{(x)}(M) \longrightarrow M\stackrel{\pi}{\longrightarrow} M/{\Gamma}_{(x)}(M) \longrightarrow 0$$
respectively imply exact sequences of $(R_0)_{\fp_0}$-modules
$$(H^{i-1}_{R_{+}}(M)_n)_{\fp_0}\longrightarrow (H^{i-1}_{R_{+}}(M/xM)_n)_{\fp_0} \longrightarrow$$
$$\longrightarrow (H^i_{R_{+}}(M/{\Gamma}_{(x)}(M))_n)_{\fp_0} \stackrel{(H^i_{R_{+}}(\phi)_n)_{\fp_0}}{\longrightarrow}
(H^i_{R_{+}}(M)_n)_{\fp_0} \longrightarrow (H^i_{R_{+}}(M/xM)_n)_{\fp_0}$$
and
$$(H^i_{R_{+}}(M)_n)_{\fp_0} \stackrel{(H^i_{R_{+}}(\pi)_n)_{\fp_0}}{\longrightarrow} (H^i_{R_{+}}(M/{\Gamma}_{(x)}(M))_n)_{\fp_0}
\longrightarrow (H^{i+1}_{R_{+}}({\Gamma}_{(x)}(M))_n)_{\fp_0}$$
for all $n\ll 0$. Keep in mind, that in the first of these sequences the first and the second but last module are of finite length
for all $n\ll 0$, whereas the second and the last module are ${\fp_0}(R_0)_{\fp_0}$-torsion-free for all $n\ll 0$. Observe further,
that in the second sequence the first module is of finite length and the last module is ${\fp_0}(R_0)_{\fp_0}$-torsion-free
for all $n\ll 0$. So there is an integer $n(x)$ such that for each $n\leq n(x)$ we have the exact sequence
$$0\longrightarrow(H^{i-1}_{R_{+}}(M/xM)_n)_{\fp_0} \longrightarrow (H^i_{R_{+}}(M/{\Gamma}_{(x)}(M))_n)_{\fp_0}
\stackrel{(H^i_{R_{+}}(\phi)_n)_{\fp_0}}{\longrightarrow} (H^i_{R_{+}}(M)_n)_{\fp_0}\longrightarrow 0$$
and the relation
$$\mathrm{Im}(H^i_{R_{+}}(\pi)_n)_{\fp_0}) = {\Gamma}_{\fp_0(R_0)_{\fp_0}}(H^i_{R_{+}}(M/{\Gamma}_{(x)}(M))_n)_{\fp_0}).$$
Thus, for all $n\leq n(x)$ the image of the composite map
$$(H^i_{R_{+}}(\pi)_n)_{\fp_0}\circ (H^i_{R_{+}}(\phi)_n)_{\fp_0}:
(H^i_{R_{+}}(M/{\Gamma}_{(x)}(M))_n)_{\fp_0} \longrightarrow (H^i_{R_{+}}(M/{\Gamma}_{(x)}(M))_n)_{\fp_0}$$
is the torsion module ${\Gamma}_{\fp_0(R_0)_{\fp_0}}((H^i_{R_{+}}(M/{\Gamma}_{(x)}(M))_n)_{\fp_0})$. As the composite map
$\pi \circ \phi: M/{\Gamma}_{(x)}(M) \longrightarrow M/{\Gamma}_{(x)}(M))$ coincides with the multiplication map
$x = x\mathrm{Id}_{M/{\Gamma}_{(x)}(M)}$ on $M/{\Gamma}_{(x)}(M)$ we end up with
$${\Gamma}_{\fp_0(R_0)_{\fp_0}}((H^i_{R_{+}}(M/{\Gamma}_{(x)}(M))_n)_{\fp_0}) = x(H^i_{R_{+}}(M/{\Gamma}_{(x)}(M))_n)_{\fp_0}, \quad
\forall n\leq n(x).$$
Now, without affecting ${\Gamma}_{(x)}(M)$ we may replace $x$ by $x^2$ and thus get the equalities
$$x(H^i_{R_{+}}(M/{\Gamma}_{(x)}(M)_n)_{\fp_0} = x^2(H^i_{R_{+}}(M/{\Gamma}_{(x)}(M))_n)_{\fp_0}$$
for all $n \leq m(x) :=\mathrm{min}\{n(x), n(x^2)\}$. Consequently, as $x \in \fp_0$ and as the $(R_0)_{\fp_0}$-modules
$(H^i_{R_{+}}(M/{\Gamma}_{(x)}(M))_n)_{\fp_0}$ are finitely generated, if follows that
$${\Gamma}_{\fp_0(R_0)_{\fp_0}}((H^i_{R_{+}}(M/{\Gamma}_{(x)}(M))_n)_{\fp_0}) = 0, \quad \forall n \ll 0.$$
Applying the functor ${\Gamma}_{\fp_0(R_0)_{\fp_0}}(\bullet)$ to the above short exact sequences and keeping in mind that the right hand
side module in these sequences is of finite length, we get the natural monomorphisms
$$0\longrightarrow (H^i_{R_{+}}(M)_n)_{\fp_0} \longrightarrow H^1_{\fp_0(R_0)_{\fp_0}}(H^{i-1}_{R_{+}}(M/xM)_n)_{\fp_0},
\quad \forall n\leq m(x).$$
It is easy to see, that these monomorphisms are the graded parts of a homomorphism of graded $R_{\fp_0}$-modules. Moreover, as
$\mathrm{dim}((R_0/xR_0)_{\fp_0}) \leq 2$ the graded $R_{\fp_0}$-module
$$H^1_{\fp(R_0)_{\fp_0}}(H^{i-1}_{R_{+}}(M/xM)_{\fp_0}) \cong
H^1_{{\fp_0}(R_0/xR_0)_{\fp_0}}(H^{i-1}_{{(R/xR)_{\fp_0}}_{+}}((M/xM)_{\fp_0}))$$
is Artinian (s. \cite{BRS} Theorem 5.10). In view of the observed monomorphisms and by Definition and Remark~\ref{4.1} (B), this
implies immediately, that the graded $R_{\fp_0}$-module $(H^i_{R_{+}}(M))_{\fp_0}$ is almost Artinian.

(b): This follows immediately from statement (a), Reminder and Remark~\ref{4.1} (D) and Remark~\ref{2.4} (B).
\end{proof}

This leads us immediately to the following observation.

\begin{cor}\label{4.5}
If $R_0$ is a domain and $i\in\mathbb{N}$ is such that the $R$-module $H^{i-1}_{R_{+}}(M)$ is almost Artinian, then the set of
all primes $\fp_0 \in \mathrm{Spec}(R_0)^{\leq 3}\setminus \overline{\fP^i(M)}$ for which the graded $R_{\fp_0}$-module
$H^i_{R_{+}}(M)_{\fp_0}$ is not almost almost Artinian as well as the set $\mathrm{Spec}(R_0)^{\leq 3} \setminus \ft^i(M)$
are both finite subsets of $\mathrm{Spec}(R_0)^{= 3}$.
\end{cor}
\begin{proof} This is immediate by Theorem~\ref{4.4} and Definition and Remark~\ref{4.1} (C).
\end{proof}



\end{document}